\documentclass[12pt]{amsart}
\usepackage{a4wide}
\usepackage[utf8]{inputenc}
\usepackage{amsmath}
\usepackage{amssymb}
\usepackage{amsfonts}
\usepackage{amsopn}
\usepackage{graphicx}
\usepackage{enumerate}
\usepackage{color}
\usepackage{mathtools}
\usepackage{mathrsfs}
\usepackage{bbm}
\usepackage{yhmath}
\usepackage{cite}
\usepackage{tikz}
\usetikzlibrary{decorations.markings}
\usepackage[colorlinks,linkcolor=blue,anchorcolor=blue,citecolor=blue]{hyperref}

\newtheorem{theorem}{Theorem}[section]
\newtheorem{proposition}[theorem]{Proposition}
\newtheorem{lemma}[theorem]{Lemma}

\newtheorem{corollary}[theorem]{Corollary}
\theoremstyle{definition}
\newtheorem{example}[theorem]{Example}
\theoremstyle{remark}

\numberwithin{equation}{section}

\newcommand{\R}{\mathbb{R}}
\newcommand{\Q}{\mathbb{Q}}

\newcommand{\N}{\mathbb{N}}

\newcommand{\f}{\infty}

\newcommand{\ii}{\mathbf{i}}
\newcommand{\jj}{\mathbf{j}}
\newcommand{\uu}{\mathbf{u}}
\newcommand{\vv}{\mathbf{v}}

\title[Self-embeddings of homogeneous self-similar sets]{Self-embeddings of homogeneous self-similar sets generated by three maps}

\author{Zhiqiang Wang}
\address[Z. Wang]{College of Mathematics and Statistics, Center of Mathematics, Key Laboratory of Nonlinear Analysis and its Applications (Ministry of Education), Chongqing University, Chongqing 401331, P.R.China}
\email{zhiqiangwzy@163.com,zqwangmath@cqu.edu.cn}


\begin{document}
\begin{abstract}
For $0< \rho < 1/3$ and $\rho \le \lambda \le 1-2\rho$,
let $E$ be the self-similar set generated by the iterated function system $$\Phi = \big\{ \varphi_1(x) = \rho x ,\; \varphi_2(x) = \rho x + \lambda, \; \varphi_3(x) = \rho x + 1- \rho \big\}.$$
All contractive similitudes $f$ with $f(E) \subset E$ are characterized: one can find $i_1, i_2, \ldots, i_n \in \{1,2,3\}$ such that \[ f(E)=\varphi_{i_1} \circ \varphi_{i_2} \circ \cdots \circ \varphi_{i_n} (E). \]
\end{abstract}

\maketitle

\section{Introduction}

In this paper, an \emph{iterated function system} (IFS) on $\R^d$ is a finite family $\Phi = \{\phi_i\}_{i=1}^m$ of contractive similitudes, where $\phi_i(x) = \rho_i O_i x + b_i$ with $0<\rho_i < 1$, $O_i$ a $d \times d$ real orthogonal matrix, and $b_i \in \R$.
According to Hutchinson \cite{Hutchinson-1981}, there exists a unique non-empty compact set $K \subset \R^d$, which is called a \emph{self-similar set}, such that
\begin{equation}\label{eq:invariant-set}
  K = \bigcup_{i=1}^m \phi_i(K).
\end{equation}
If $\rho_i O_i \equiv \rho_{\Phi} O_{\Phi}$ for all $1 \le i \le m$, then we say the IFS $\Phi$ is \emph{homogeneous}, and the corresponding set $K$ is called a homogeneous self-similar set.
The IFS $\Phi$ is said to satisfy the \emph{open set condition} (OSC) if there is a non-empty open set $V \subset \R^d$ such that the sets $\phi_i(V),\; 1\le i\le m,$ are pairwise disjoint subsets of $V$.
In this case, the Hausdorff dimension of $K$ is given by the value $s>0$ satisfying the equation $\sum_{i=1}^m \rho_i^s = 1$ \cite{Hutchinson-1981}.
If the sets $\phi_i(K),\; 1\le i\le m,$ are pairwise disjoint, then we say the IFS $\Phi$ satisfies the \emph{strong separation condition} (SSC).
Clearly, the SSC implies the OSC, but not vice versa.
We always assume that a self-similar is not a singleton.

Let $K$ be a self-similar set generated by an IFS $\Phi = \{\phi_i\}_{i=1}^m$ in $\R^d$.
Let $\Omega_m^*$ denote the set of all finite words over the alphabet $\{1,2,\ldots,m\}$ including the empty word $\emptyset$.
For $\ii = i_1 i_2 \ldots i_n \in \Omega_m^*$, define \[ \phi_{\ii}: = \phi_{i_1} \circ \phi_{i_2} \circ \cdots \circ \phi_{i_n},\]
where $\phi_{\emptyset}$ denotes the identity map.
By (\ref{eq:invariant-set}), the set $K$ consists of many small similar copies of itself, and we have \[ \phi_{\ii}(K)\subset K\quad \forall \ii\in \Omega_m^*. \]
A natural question is to characterize all similitudes $f$ such that $f(K) \subset K$.
This question is related to the generating IFS problem \cite{Feng-Wang-2009,Deng-Lau-2013,Deng-Lau-2017} and the group of isometries of fractals \cite{Moran-2012,Deng-Yao-2022}.


Let $f$ be a similitude on $\R^d$ such that $f(K) \subset K$.
In \cite{Feng-Wang-2009}, Feng and Wang proved the so-called \emph{logarithmic commensurability theorem}: if $\Phi$ is a homogeneous IFS in $\R$ satisfying the OSC and $\dim_H K < 1$, then we have $\log \rho_f / \log \rho_\Phi \in \Q$, where $\rho_f$ is the similitude ratio of $f$ and $\rho_\Phi$ is the common contraction ratio of $\Phi$.
This result was generalized to the higher dimensional cases by Elekes, Keleti, and M\'{a}th\'{e} \cite{Elekes-Keleti-Mathe-2010}. They proved that if the IFS $\Phi$ in $\R^d$ satisfies the SSC, then there exist a positive integer $k \in \N$ and two finite words $\ii,\jj \in \Omega_m^*$ such that \[ f^k \circ \phi_{\ii} = \phi_{\jj}. \]
Under a much stronger separation condition, Yao \cite{Yao-2017} showed that there exists a finite word $\ii \in \Omega_m^*$ such that \[ f(K) = \phi_{\ii}(K). \]
Recently, Xiao \cite{Xiao-2024} showed the relative openness of $f(K)$ for the homogeneous IFS satisfying the SSC, that is, there exists $\mathcal{I} \subset \Omega_m^*$ such that \[ f(K) = \bigcup_{\ii \in \mathcal{I}} \phi_{\ii}(K). \]
In the same paper, Xiao \cite{Xiao-2024} also constructed a counterexample in the plane when the orthogonal parts vary.
It is worth noting that Algom and Hochman \cite{Algom-Hochman-2019} have studied self-embeddings of Bedford-McMullen carpets.

In \cite{Moran-2012}, Mor\'{a}n studied the group of isometries of a self-similar set, and proved that if the IFS $\Phi$ satisfies the SSC, then the group \[ \mathcal{I}_K := \big\{ \sigma: \; \sigma\;\text{is an isometry and}\; \sigma(K) = K  \big\} \] is finite.
Furthermore, by Proposition 2.2 in \cite{Deng-Lau-2017}, one can show the same conclusion holds if $K \subset \R^d$ is a totally disconnected compact set and spans $\R^d$.
By Lemma 4.8 in \cite{Elekes-Keleti-Mathe-2010}, one can deduce that if the IFS $\Phi$ satisfies the SSC, then the semi-group \[ \mathcal{S}_K :=\big\{ f: \; f\;\text{is a similitude and}\; f(K) \subset K \big\} \] is finitely generated, and thus, is countable.
Without separation conditions, Hochman \cite{Hochman-2017} showed that if $K$ is a self-similar set in $\R$ with $\dim_H K < 1$, then $\dim_H \mathcal{S}_K =0$ where the set $\mathcal{S}_K$ is viewed as a subset of $\R^2$.
Algom \cite{Algom-2020} proved a similar result in the plane under a mild condition.

We expect that each $f\in \mathcal{S}_K$ could be written as $f = \phi_{\ii} \circ \sigma$ where $\ii \in \Omega_m^*$ and $\sigma \in \mathcal{I}_K$.
This expectation does not come true even under the SSC assumption (see Example \ref{example}), but it has been confirmed in a few particular cases such as a class of totally self-similar sets \cite{Yao-Li-2015,Dajani-Kong-Yao-2019}, the Vicsek snowflake and the Koch curve \cite{Yao-Li-2016}, and a class of self-similar sets with complete overlap but not totally self-similar \cite{Kong-Yao-2021}.

For $0< \rho < 1/3$ and $\rho \le \lambda \le 1-2\rho$,
let $E_{\rho,\lambda}$ be the self-similar set generated by the IFS $$\Phi_{\rho,\lambda} := \big\{ \varphi_1(x) = \rho x ,\; \varphi_2(x) = \rho x + \lambda, \; \varphi_3(x) = \rho x + 1- \rho \big\}.$$
Note that the convex hull of $E_{\rho,\lambda}$ is the interval $[0,1]$, and the IFS $\Phi_{\rho,\lambda}$ always satisfies the OSC with the open set $V =(0,1)$.
The main result of this paper is to characterize all similitudes $f$ such that $f(E_{\rho,\lambda}) \subset E_{\rho,\lambda}$.

\begin{theorem}\label{thm:three-map}
  Let $f$ be a contractive similitude on $\R$ such that $f(E_{\rho,\lambda}) \subset E_{\rho,\lambda}$.

  {\rm(i)} If $\lambda \ne (1-\rho)/2$, then there exists $\ii \in \Omega_3^*$ such that $f(x) = \varphi_{\ii}(x)$.

  {\rm(ii)} If $\lambda = (1-\rho)/2$, then there exists $\ii \in \Omega_3^*$ such that $f(x) = \varphi_{\ii}(x)$ or $f(x) = \varphi_{\ii}(1-x)$.
\end{theorem}

When $\lambda = (1-\rho)/2$, the set $E_{\rho,\lambda}$ is symmetric, i.e., $1-E_{\rho,\lambda} = E_{\rho,\lambda}$, and the gap between $\varphi_1([0,1])$ and $\varphi_2([0,1])$ is equal to that between $\varphi_2([0,1])$ and $\varphi_3([0,1])$. The part (ii) of Theorem \ref{thm:three-map} follows directly from the following more general theorem.

\begin{theorem}\label{thm:equal-gap}
  Given an integer $m \ge 2$ and $0< \rho_1, \rho_2, \ldots, \rho_m < 1$ such that $\rho_1 + \rho_2 + \cdots + \rho_m < 1$, let $K$ be the self-similar set generated by the IFS \[ \Big\{ \phi_i(x) = \rho_i x + \sum_{k=1}^{i-1} \rho_k + (i-1) \gamma \Big\}_{i=1}^m, \] where $\gamma := (1-\rho_1 -\rho_2 - \cdots - \rho_m)/(m-1)$.
  Let $f$ be a contractive similitude on $\R$ such that $f(K) \subset K$.
  Then there exists $\ii \in \Omega_m^*$ such that \[ f(K) = \phi_{\ii}(K). \]

  {\rm(i)} If $1-K \ne K$, then we have $f(x) = \phi_{\ii}(x)$;

  {\rm(ii)} If $1-K = K$, then we have $f(x) = \phi_{\ii}(x)$ or $f(x) = \phi_{\ii}(1-x)$.
\end{theorem}

Applying Theorem \ref{thm:equal-gap} to two typical self-similar sets, we obtain the following corollary, which is useful to study the self-similarity of unions of self-similar set and its translations.

\begin{corollary}
  Let $f$ be a contractive similitude on $\R$.

  {\rm(i)} Given $\alpha, \beta > 0$ such that $\alpha \ne \beta$ and $\alpha + \beta < 1$, let $K_{\alpha,\beta}$ be the self-similar set generated by the IFS \[ \big\{ \phi_1(x) = \alpha x, \; \phi_2(x) = \beta x + 1-\beta \big\}. \]
  If $f(K_{\alpha,\beta}) \subset K_{\alpha,\beta}$, then there exists $\ii \in \Omega_2^*$ such that $ f(x) = \phi_{\ii}(x)$.

  {\rm(ii)} Given an integer $m \ge 2$ and $0 < \beta < 1/m$, let $\Gamma_{\beta,m}$ be the homogeneous self-similar set generated by the IFS \[ \Big\{ \phi_i(x) = \beta x + \frac{(i-1)(1-\beta)}{m-1} \Big\}_{i=1}^m .\]
  If $f(\Gamma_{\beta,m}) \subset \Gamma_{\beta,m}$, then there exists $\ii \in \Omega_m^*$ such that $f(x) = \phi_{\ii}(x)$ or $f(x) = \phi_{\ii}(1-x)$.
\end{corollary}

Finally, we give an example to illustrate that Theorem \ref{thm:three-map} cannot be extended to general homogeneous self-similar sets generated by four maps.
The following example appeared in \cite[Example 6.2]{Feng-Wang-2009}, and the general situation was also investigated in \cite[Theorem 6.2]{Elekes-Keleti-Mathe-2010} and \cite[Theorem 1.5]{Yao-2015}.

\begin{example}\label{example}
  Let $K$ be the self-similar set generated by the IFS \[ \Phi=\Big\{ \phi_1(x) = \frac{x}{10}, \; \phi_2(x) = \frac{x+1}{10},\; \phi_3(x) = \frac{x+5}{10},\; \phi_4(x) = \frac{x+6}{10} \Big\}. \]
  Note that $10K = K +\{0,1,5,6\}$.
  Then we have \[ \phi_{13}(K) \cup \phi_{14}(K) \cup \phi_{21}(K)\cup \phi_{22}(K) = \frac{K + \{5,6,10,11\}}{100} = \frac{K}{10} + \frac{1}{20} \subset K. \]
  Let $g_1(x) = x /10 + 1/20$ and $g_2(x) = x /10 + 11/20$.
  Then we have $g_1(K) \subset K$ and $g_2(K) \subset K$, but the maps $g_1(x)$ and $g_2(x)$ cannot be derived from the IFS $\Phi$. Note that $K$ is symmetric, i.e., $2/3 - K = K$.
  Let $\sigma(x) = 2/3- x$.
  By Theorem 1.5 in \cite{Yao-2015}, each contractive similitude $f$ satisfying $f(K) \subset K$ has one of the following forms \[ \phi_{\ii},\; \phi_{\ii}\circ\sigma,\; \phi_{\ii} \circ g_1, \; \phi_{\ii} \circ g_1 \circ \sigma,\;  \phi_{\ii} \circ g_2, \; \phi_{\ii} \circ g_2\circ \sigma, \]
  where $\ii \in \Omega_4^*$.
\end{example}

The rest of paper is organized as follows.
In Section \ref{sec:pre}, we first introduce some notations and give several lemmas. Then we prove Theorem \ref{thm:equal-gap}, which implies Theorem \ref{thm:three-map} (ii).
At the end of Section \ref{sec:pre}, we reduce Theorem \ref{thm:three-map} (i) to the case $\rho \le \lambda < (1-\rho)/2$.
We will prove Theorem \ref{thm:three-map} (i) for $\rho < \lambda < (1-\rho)/2$ and $\lambda = \rho$ in Section \ref{sec:SSC} and \ref{sec:OSC}, respectively.

\section{Preliminaries}\label{sec:pre}

In this paper, all compact sets we refer to contain at least two points.
For a compact set $K \subset \R$, let $\mathrm{Conv}(K)$ denote the convex hull of $K$, which is the smallest closed interval containing $K$.
Then $\mathrm{Conv}(K) \setminus K$ can be written as a union of disjoint open intervals, each of which is called a \emph{gap} of $K$.
Define $\mathfrak{g}(K)$ to be the largest length of gaps of $K$.
For a compact interval $K\subset \R$, we define $\mathfrak{g}(K)$ to be $0$.

For a finite word $\ii = i_1 i_2 \ldots i_n \in \Omega^*$, let $|\ii|$ denote the length of $\ii$.
We also use $|I|$ to denote the length of an bounded interval $I \subset \R$.

For a non-empty subset $F \subset \R$ and $\delta>0$, the \emph{$\delta$-neighborhood} of $F$ is defined by
\[ F^\delta:=\big\{ y \in \R: \;\text{there exists}\; x\in F\;\text{such that}\; |y-x|<\delta \big\}. \]
For two non-empty subsets $F_1,F_2 \subset \R$, define \[\mathrm{dist}(F_1,F_2):= \inf\big\{ |x-y|: x\in F_1, y \in F_2 \big\}.\]
It is easy to verify that
\begin{equation}\label{eq:dist}
  \mathrm{dist}(F_1,F_2) = \inf\big\{\delta>0: F_1^{\delta/2} \cap F_2^{\delta/2} \ne \emptyset\big\},
\end{equation}
and for a compact set $K \subset \R$,
\begin{equation}\label{eq:g-K}
  \mathfrak{g}(K) = \inf\big\{ \delta>0: K^{\delta/2} \;\text{is an open interval} \big\}.
\end{equation}

Next, we give several lemmas which will be used in the subsequent proof.

\begin{lemma}\label{gap-lemma}
  Let $F,K\subset \R$ be two compact sets with $F \subset K$.
  If the set $K$ can be written as $K= K_1 \cup K_2 \cup \cdots \cup K_\ell$ with $K_i \ne \emptyset$ for all $1\le i \le \ell$ and \[ \mathfrak{g}(F) < \min_{1\le i < j \le \ell}\mathrm{dist}(K_i, K_j), \]
  then there exists $i \in \{1,2,\ldots,\ell\}$ such that $F\subset K_i$.
\end{lemma}
\begin{proof}
  Suppose on the contrary that there exist $i_1 \ne i_2 \in \{1,2,\ldots,\ell\}$ such that $F \cap K_{i_1} \ne \emptyset$ and $F\cap K_{i_2} \ne \emptyset$.
  Let $F_1= F \cap K_{i_1}$ and $F_2 = F \setminus K_{i_1}$.
  Then we have $F_1 \ne \emptyset$, $F_2\ne \emptyset$, and
  \[ \mathrm{dist}(F_1,F_2) \ge \min_{i \ne i_i} \mathrm{dist}(K_{i_1},K_i) > \mathfrak{g}(F).\]
  We can choose $\delta >0$ such that \[ \mathfrak{g}(F) < \delta < \mathrm{dist}(F_1,F_2). \]
  By (\ref{eq:g-K}), the set $F^{\delta/2} $ is an open interval.
  Since $F = F_1 \cup F_2$, we have \[F^{\delta/2} = F_1^{\delta/2} \cup F_2^{\delta/2}.\]
  Note that $F_1^{\delta/2}$ and $F_2^{\delta/2}$ are two non-empty open sets.
  Thus, the fact that $F^{\delta/2} $ is an open interval implies $F_1^{\delta/2} \cap F_2^{\delta/2} \ne \emptyset$.
  It follows from (\ref{eq:dist}) that $\mathrm{dist}(F_1,F_2) \le \delta$.
  This leads to a contradiction.
  We complete the proof.
\end{proof}

\begin{lemma}\label{lemma:dichotomy}
  Suppose that $\rho \le \lambda < (1-\rho)/2$, and write $E=E_{\rho,\lambda}$.
  Let $f(x) = r x + t$ with $r \ne 0$ and $t\in \R$.

  {\rm(i)} If $|r|<1$ and $f(E) \subset E$, then we have $f(E) \subset \varphi_1(E) \cup \varphi_2(E)$ or $f(E) \subset \varphi_3(E)$.

  {\rm(ii)} If $f\big( \varphi_1(E) \cup \varphi_2(E) \big) \subset E$, then we have $f\big( \varphi_1(E) \cup \varphi_2(E) \big) \subset \varphi_1(E) \cup \varphi_2(E)$ or $f\big( \varphi_1(E) \cup \varphi_2(E) \big) \subset \varphi_3(E)$.
\end{lemma}
\begin{proof}
  Let $E_1 = \varphi_1(E) \cup \varphi_2(E)$ and $E_2=\varphi_3(E)$. Then we have $E = E_1 \cup E_2$.

  (i) Since $\rho \le \lambda < (1-\rho)/2$ and $|r| < 1$, we have \[ \mathrm{dist}(E_1, E_2)= \mathfrak{g}(E) > \mathfrak{g}\big( f(E) \big). \]
  The desired result follows directly from Lemma \ref{gap-lemma}.

  (ii) Since $f\big( \varphi_1(E) \cup \varphi_2(E) \big) \subset E$, we have $|r|\rho < 1$.
  By applying (i) for the map $f\circ \varphi_1$, we have $f\circ \varphi_1(E) \subset E_1$ or $f\circ \varphi_1(E) \subset E_2$.
  Similarly, we also have $f\circ \varphi_2(E) \subset E_1$ or $f\circ \varphi_2(E) \subset E_2$.

  If $f\big( \varphi_1(E) \cup \varphi_2(E) \big) \cap E_2 = \emptyset$, then we conclude directly that $f\big( \varphi_1(E) \cup \varphi_2(E) \big) \subset E_1$.
  If $f\big( \varphi_1(E) \cup \varphi_2(E) \big) \cap E_2 \ne \emptyset$, then there must be $i \in \{1,2\}$ such that
  $f\circ \varphi_i(E) \subset E_2 = \varphi_3(E)$.
  It follows that $|r|\le 1$.
  Since $\rho \le \lambda < (1-\rho)/2$, we have
  \[ \mathfrak{g} \big( f\big( \varphi_1(E) \cup \varphi_2(E) \big) \big) \le \mathfrak{g} \big( \varphi_1(E) \cup \varphi_2(E) \big) < \mathfrak{g}(E) = \mathrm{dist}(E_1, E_2).\]
  By Lemma \ref{gap-lemma}, we conclude that $f\big( \varphi_1(E) \cup \varphi_2(E) \big) \subset E_2$.
\end{proof}

A set $K \subset \R$ is called \emph{symmetric} if there exists $t_0 \in \R$ such that $t_0- K = K$.
\begin{lemma}\label{lemma:symmetric}
  Suppose that $\rho \le \lambda < (1-\rho)/2$.
  Then the sets $E_{\rho,\lambda}$ and $\varphi_1(E_{\rho,\lambda}) \cup \varphi_2(E_{\rho,\lambda})$ are not symmetric.
\end{lemma}
\begin{proof}
  For $\delta = 1-\rho - 2\lambda$, we have $E_{\rho,\lambda} \cap (\lambda, \lambda +\delta) \ne \emptyset$ but $(1-E_{\rho,\lambda}) \cap (\lambda,\lambda+\delta) =\emptyset$. It follows that $1-E_{\rho,\lambda} \ne E_{\rho,\lambda}$, and hence, the set $E_{\rho,\lambda}$ is not symmetric.

Suppose that the set $\varphi_1(E_{\rho,\lambda}) \cup \varphi_2(E_{\rho,\lambda})$ is symmetric.
Then $(\lambda+\rho) - \big( \varphi_1(E_{\rho,\lambda}) \cup \varphi_2(E_{\rho,\lambda}) \big) = \varphi_1(E_{\rho,\lambda}) \cup \varphi_2(E_{\rho,\lambda})$.
It follows that
\begin{align*}
  \rho E_{\rho,\lambda} & = \varphi_1(E_{\rho,\lambda}) = \big( \varphi_1(E_{\rho,\lambda}) \cup \varphi_2(E_{\rho,\lambda}) \big) \cap [0,\rho]  \\
  & = \big( (\lambda+\rho) - ( \varphi_1(E_{\rho,\lambda}) \cup \varphi_2(E_{\rho,\lambda})) \big)  \cap [0,\rho] \\
  & = (\lambda+\rho) - \varphi_2(E_{\rho,\lambda}) \\
  & = \rho(1-E_{\rho,\lambda}).
\end{align*}
That is, $1-E_{\rho,\lambda}=E_{\rho,\lambda}$, a contradiction.
Thus, the set $\varphi_1(E_{\rho,\lambda}) \cup \varphi_2(E_{\rho,\lambda})$ is not symmetric.
\end{proof}

\begin{lemma}\label{lemma:subset-to-equal}
  Let $f(x) = rx + t$ with $|r|=1$ and $t \in \R$.
  If $K \subset \R$ is non-empty and $f(K) \subset K$, then we have $f(K) = K$.
\end{lemma}
\begin{proof}
  Note that $f^2(x) = x + rt + t$ and $f^2(K) \subset f(K) \subset K$. Then we must have $f^2(x) = x$ and $f^2(K) = K$, which implies that $f(K) = K$.
\end{proof}

Now, we can prove Theorem \ref{thm:equal-gap} by using Lemma \ref{gap-lemma}.

\begin{proof}[Proof of Theorem \ref{thm:equal-gap}]
  Note that $\mathrm{Conv}(K) = [0,1]$, and the gaps between \[ \phi_1([0,1]),\; \phi_2([0,1]),\; \ldots,\; \phi_m([0,1]) \]
  are equal to $\gamma$. Then we have $\mathfrak{g}(K) = \gamma$.
  For $1 \le i \le m$, let $K_i = \phi_{i}(K)$.
  Then we have \[ \min_{1\le i < j \le \ell}\mathrm{dist}(K_i, K_j) = \gamma = \mathfrak{g}(K) > \mathfrak{g}\big( f(K) \big).  \]
  By Lemma \ref{gap-lemma}, there exists $i_1 \in \{1,2,\ldots,m\}$ such that $f(K) \subset \phi_{i_1}(K)$.
  Let $\ii \in \Omega_m^*$ be the longest word such that $f(K) \subset \phi_{\ii}(K)$.
  If the similitude $\phi_{\ii}^{-1} \circ f$ is contractive, then by the same argument for $\phi_{\ii}^{-1} \circ f$, there exists $i_0 \in \{1,2,\ldots,m\}$ such that $\phi_{\ii}^{-1} \circ f(K) \subset \phi_{i_0}(K)$, i.e., $f(K) \subset \phi_{\ii i_0}(K)$, which contradicts the longest length of $\ii$.
  Thus, the similitude ratio of $\phi_{\ii}^{-1} \circ f$ is $1$.
  By Lemma \ref{lemma:subset-to-equal}, we conclude that $\phi_{\ii}^{-1} \circ f(K)= K$, i.e., $f(K)= \phi_{\ii}(K)$.
\end{proof}

In the end of this section, we reduce the part (i) of Theorem \ref{thm:three-map}.

\begin{lemma}\label{lemma:reduce}
  It suffices to prove Theorem \ref{thm:three-map} (i) only for $\rho\le \lambda < (1-\rho)/2$.
\end{lemma}
\begin{proof}
  Suppose that we have proved Theorem \ref{thm:three-map} (i) for $\rho\le \lambda < (1-\rho)/2$.
  In the following, we assume $(1-\rho)/2 < \lambda \le 1-2\rho$ and write $E = E_{\rho,\lambda}$.
  Let $f(x) = r x + t$ with $0< |r| < 1$ and $t \in \R$ such that $f(E) \subset E$.

  Let $F = 1- E$. Then the set $F$ is the self-similar set generated by the IFS \[ \big\{ \phi_1(x) = \rho x,\;\phi_2(x) = \rho x + 1-\rho - \lambda,\;\phi_3(x) = \rho x + 1- \rho \big\}. \]
  Note that $\rho\le 1-\rho - \lambda < (1-\rho)/2$.
  Let $g(x) = f(x) + 1 - r - 2 t = r x + 1- r - t$.
  Then we have \[ g(F) = r(1-E) + 1- r - t = 1- (rE+t) \subset 1-E = F. \]
  By the assumption, there exists $\ii = i_1 i_2 \ldots i_n \in \Omega_3^*$ such that $g(x) = \phi_{\ii}(x)$.
  That is, $r = \rho^n$ and $1-r - t = \phi_{\ii}(0)$.

  Let $\jj = j_1 j_2 \ldots j_n \in \Omega_3^*$ where $j_k = 4 - i_k$ for all $1 \le k \le n$.
  Note that $(1-\rho) - \phi_{i_k}(0) = \varphi_{j_k}(0)$ for all $1 \le k \le n$.
  Thus we have \[ t = (1-r) - \phi_{\ii}(0) = \sum_{k=1}^{n} (1-\rho)\rho^{k-1} - \sum_{k=1}^{n} \phi_{i_k}(0) \rho^{k-1} = \sum_{k=1}^{n} \varphi_{j_k}(0) \rho^{k-1} =\varphi_{\jj}(0). \]
  It follows that $f(x) = \rho^n x + \varphi_{\jj}(0) = \varphi_{\jj}(x)$, as desired.
\end{proof}

The part (ii) of Theorem \ref{thm:three-map} follows directly from Theorem \ref{thm:equal-gap}.
We will prove Theorem \ref{thm:three-map} (i) for $\rho< \lambda < (1-\rho)/2$ and $\lambda = \rho$ in Section \ref{sec:SSC} and \ref{sec:OSC}, respectively.
It this were done, by Lemma \ref{lemma:reduce} we would be able to prove Theorem \ref{thm:three-map} (i), and hence, we would complete the proof of Theorem \ref{thm:three-map}.

\section{Self-embeddings for $\rho< \lambda < (1-\rho)/2$}\label{sec:SSC}

In this section we will prove Theorem \ref{thm:three-map} (i) for $\rho< \lambda < (1-\rho)/2$.
We always assume that $\rho < \lambda < (1-\rho)/2$ and write $E=E_{\rho,\lambda}$ and $\Omega^* = \Omega_3^*$.
Let $G_1 := (\rho,\lambda)$ and $G_2 := (\lambda+\rho, 1-\rho)$. 
\begin{figure}[h!]
\begin{center}
\begin{tikzpicture}[
    scale=10,
    axis/.style={very thick, ->},
    important line/.style={very thick},
    every node/.style={color=black}
    ]

    \draw[{|-|}, important line] (0, 0)--(1, 0);
    \node[above,scale=1pt]at(0,0.02){$0$};
    \node[above,scale=1pt]at(1,0.02){$1$};

    \draw[{|-|}, important line] (0, -0.05)--(0.25, -0.05);
    \node[above,scale=1pt]at(0.12,-0.13){$\varphi_1([0,1])$};
    \draw[{|-|}, important line] (0.35, -0.05)--(0.6, -0.05);
    \node[above,scale=1pt]at(0.47,-0.13){$\varphi_2([0,1])$};
    \draw[{|-|}, important line] (0.75, -0.05)--(1, -0.05);
    \node[above,scale=1pt]at(0.87,-0.13){$\varphi_3([0,1])$};

    \node[above,scale=1pt]at(0.3,-0.085){$G_1$};
    \node[above,scale=1pt]at(0.675,-0.085){$G_2$};
\end{tikzpicture}
\end{center}
\caption{The first iteration for $E_{\rho,\lambda}$ when $\rho< \lambda < (1-\rho)/2$}\label{fig:1}
\end{figure}
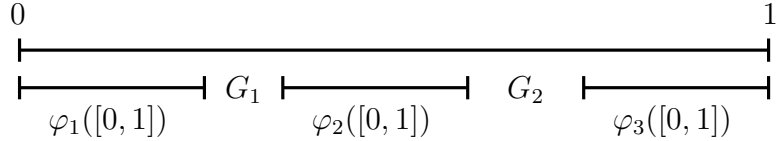

It is easy to verify that $\mathrm{Conv}(E) = [0,1]$ and \[ [0,1] \setminus E = \bigcup_{\ii \in \Omega^*} \big( \varphi_{\ii}(G_1) \cup \varphi_{\ii}(G_2) \big). \]
Since $\rho < \lambda < (1-\rho)/2$, we have \[ \mathfrak{g}(E) = |G_2| > |G_1| >0.\]

\begin{proposition}\label{prop:positive-1}
  Let $f(x) = r x + t$ with $0 < r < 1$ and $t \in \R$. If $f(E) \subset E$, then there exists $i \in \{1,2,3\}$ such that $f(E) \subset \varphi_{i}(E)$.
\end{proposition}
\begin{proof}
  Suppose on the contrary that $f(E) \not\subset \varphi_i(E)$ for all $i \in \{1,2,3\}$.
  By Lemma \ref{lemma:dichotomy} (i), we have
  \begin{equation}\label{eq:subset-f}
    f(E) \subset \varphi_1(E) \cup \varphi_2(E),
  \end{equation}
  and \[ f(E) \cap \varphi_1(E) \ne \emptyset, \; f(E) \cap \varphi_2(E) \ne \emptyset. \]
  Then the open interval $G_1$ is contained in gaps of $f(E)$.
  That is, $$ G_1 \subset \bigcup_{\ii \in \Omega^*} \big( f\circ \varphi_{\ii} (G_1) \cup f\circ \varphi_{\ii} (G_2) \big).$$
  Note that $|f\circ \varphi_{\ii} (G_1)| < |G_1|$ for any $\ii \in \Omega^*$.
  Thus, there exists $\jj \in \Omega^*$ such that
  \begin{equation}\label{eq:condition-3}
   G_1 \subset f \circ \varphi_{\jj} (G_2).
  \end{equation}
  This implies that $f\circ \varphi_{\jj}(\rho+\lambda) \le \rho$.
  It follows that
  $$f\circ \varphi_{\jj}\big( \varphi_1(E) \cup \varphi_2(E) \big) \subset f(E) \cap (-\f,\rho] \subset \varphi_1(E),$$
  where the last inclusion follows from (\ref{eq:subset-f}).
  Choose the longest word $\ii =i_1 i_2 \ldots i_n\in \Omega^*$ with $i_1 = 1$ such that
  $$f\circ \varphi_{\jj}\big( \varphi_1(E) \cup \varphi_2(E) \big) \subset \varphi_{\ii}(E).$$
  Then we have $f\circ \varphi_{\jj}\big( \varphi_1(E) \cup \varphi_2(E) \big) \not\subset \varphi_{\ii 3}(E)$.
  Applying Lemma \ref{lemma:dichotomy} (ii) for the map $\varphi_{\ii}^{-1} \circ f \circ \varphi_{\jj}$, we obtain
  \[ f\circ \varphi_{\jj}\big( \varphi_1(E) \cup \varphi_2(E) \big) \subset \varphi_{\ii}\big( \varphi_1(E) \cup \varphi_2(E) \big). \]
  Note that $\mathrm{Conv}\big( \varphi_1(E) \cup \varphi_2(E) \big) = [0,\rho+\lambda]$.
  Thus we obtain $f\circ \varphi_{\jj}( [0,\rho+\lambda] ) \subset \varphi_{\ii}( [0,\rho+\lambda] )$, which implies that $f \circ \varphi_{\jj}([0,1]) \subset \varphi_{\ii}([0,1]) \subset \varphi_1([0,1])$.
  Together with (\ref{eq:condition-3}), we have
  \[ G_1 \subset f \circ \varphi_{\jj}(G_2) \subset f \circ \varphi_{\jj}([0,1]) \subset \varphi_1([0,1]), \]
  which leads to a contradiction.
  The proof is completed.
\end{proof}

\begin{proposition}\label{prop:no-negative-1}
  For any $0 < r < 1$ and for any $t \in \R$, we have $$ -r \big( \varphi_1(E) \cup \varphi_2(E) \big) + t \not\subset E.$$
\end{proposition}
\begin{proof}
  Suppose on the contrary that there exist $0< r < 1$ and $t_0 \in \R$ such that \[ -r \big( \varphi_1(E) \cup \varphi_2(E) \big) + t_0 \subset E. \]
  Let $$\ell = \max\big\{ n \in \N \cup \{0\}: \;\text{there exists}\; t \in \R \;\text{such that}\; -r\rho^{-n}\big( \varphi_1(E) \cup \varphi_2(E) \big) + t \subset E \big\}.$$
  Then we can find $t_1 \in \R$ such that $ -r\rho^{-\ell}\big( \varphi_1(E) \cup \varphi_2(E) \big) + t_1 \subset E$.
  Write $r' = r \rho^{-\ell}$ and $f(x) = - r' x + t_1$.
  Then we have $$f\big( \varphi_1(E) \cup \varphi_2(E) \big) \subset E.$$
  Due to the maximality of $\ell$, for any two words $\uu,\vv\in \Omega^*$ with $|\vv| > |\uu|$, we have
  \begin{equation}\label{eq:maximality}
    f\circ \varphi_{\uu} \big( \varphi_1(E) \cup \varphi_2(E) \big) \not\subset \varphi_{\vv} (E).
  \end{equation}
  It follows that
  \begin{equation}\label{eq:not-subset}
    f\big( \varphi_1(E) \cup \varphi_2(E) \big) \not\subset \varphi_{i}(E)\quad \forall i \in \{1,2,3\}.
  \end{equation}
  By Lemma \ref{lemma:dichotomy} (ii), we obtain
  \begin{equation}\label{eq:subset-f-1-2}
    f\big( \varphi_1(E) \cup \varphi_2(E) \big) \subset \varphi_1(E) \cup \varphi_2(E).
  \end{equation}
  It follows that $r' \le 1$.
  If $r' = 1$, then by Lemma \ref{lemma:subset-to-equal} we have $- \big( \varphi_1(E) \cup \varphi_2(E) \big) + t_1 =  \varphi_1(E) \cup \varphi_2(E) $.
  This means that the set $\varphi_1(E) \cup \varphi_2(E)$ is symmetric, which contradicts Lemma \ref{lemma:symmetric}.
  Thus we have $r' < 1$.

  Note by (\ref{eq:not-subset}) and (\ref{eq:subset-f-1-2}) that
  \begin{equation}\label{eq:condition-7}
    f\big( \varphi_1(E) \cup \varphi_2(E) \big) \cap \varphi_{1}(E) \ne \emptyset \text{ and } f\big( \varphi_1(E) \cup \varphi_2(E) \big) \cap \varphi_{2}(E) \ne \emptyset.
  \end{equation}
  The open interval $G_1$ must be contained in gaps of $f\big( \varphi_1(E) \cup \varphi_2(E) \big)$.
  That is, $$ G_1 \subset f \circ \varphi_1([0,1] \setminus E) \cup f(G_1) \cup f \circ \varphi_2([0,1]\setminus E).$$
  Note that $|f(G_1)| = r' |G_1| < |G_1|$.
  So, we have two cases:
  \begin{itemize}
    \item[\rm{(i)}] $G_1 \subset f\circ \varphi_1([0,1] \setminus E)$;
    \item[\rm{(ii)}] $G_1 \subset f\circ \varphi_2([0,1] \setminus E)$.
  \end{itemize}

  \textbf{Case (i)}: $G_1 \subset f\circ \varphi_1([0,1] \setminus E)$.
  It follows that $G_1 \subset f\circ \varphi_1 \big( (0,1) \big) = \big( f(\rho), f(0) \big)$, which implies that $f(\lambda) < f(\rho) \le \rho$.
  Note by (\ref{eq:subset-f-1-2}) that $f\circ \varphi_2(E) \subset \varphi_{1}(E) \cup \varphi_{2}(E)$.
  We have
  \begin{align*}
    f\circ \varphi_2(E) & \subset \big(\varphi_{1}(E) \cup \varphi_{2}(E) \big) \cap (-\f, f(\lambda)] \\
    & \subset \big(\varphi_{1}(E) \cup \varphi_{2}(E) \big) \cap (-\f, \rho] \\
    & = \varphi_{1}(E).
  \end{align*}
  By (\ref{eq:maximality}), we have $ f\circ \varphi_2(E) \not\subset \varphi_{1 3}(E)$.
  Applying Lemma \ref{lemma:dichotomy} (i) for the map $\varphi_{1}^{-1} \circ f \circ \varphi_2$, we conclude that $$f\circ \varphi_2(E) \subset \varphi_{1} \big( \varphi_1(E) \cup \varphi_2(E) \big).$$
  Together with (\ref{eq:condition-7}), we obtain that the open interval $\varphi_{1}(G_2)$ is contained in gaps of $f\big( \varphi_1(E) \cup \varphi_2(E) \big)$.
  Note that $$\mathfrak{g}\big( f\circ \varphi_1(E) \big) = \mathfrak{g}\big( f\circ \varphi_2(E) \big)  = r' \rho |G_2| < \rho |G_2| = |\varphi_{1}(G_2)|. $$
  Thus we have $$\varphi_{1}(G_2) \subset f(G_1).$$
  It follows that $\rho |G_2| \le r' |G_1| < |G_1|$.
  Since $G_1 \subset f\circ \varphi_1([0,1] \setminus E)$, we have
  \[ |G_1| \le \mathfrak{g}\big( f\circ \varphi_1(E) \big) = r'\rho|G_2| < \rho|G_2|, \]
  which leads to a contradiction.

  \textbf{Case (ii):} $G_1 \subset f\circ \varphi_2([0,1] \setminus E)$.
  It follows that $G_1 \subset f\circ \varphi_2 \big( (0,1) \big) = \big( f(\rho+\lambda), f(\lambda) \big)$, which implies that
  \begin{equation}\label{eq:condition-14}
    f(\rho) > f(\lambda) \ge \lambda.
  \end{equation}
  Note by (\ref{eq:subset-f-1-2}) that $f\circ \varphi_1(E) \subset \varphi_{1}(E) \cup \varphi_{2}(E)$.
  We have
  \begin{align*}
    f\circ \varphi_1(E) & \subset \big( \varphi_{1}(E) \cup \varphi_{2}(E) \big) \cap [f(\rho),+\f) \\
    & \subset \big( \varphi_{1}(E) \cup \varphi_{2}(E) \big) \cap [\lambda,+\f) \\
    & = \varphi_{2}(E).
  \end{align*}

  We claim
  \begin{itemize}
    \item $f\circ \varphi_{1 3^n}(E) \subset \varphi_{2 1^n}(E)$ for all $n \in \N \cup \{0\}$.
  \end{itemize}
    Obviously, the claim holds for $n=0$.
    Suppose that $f\circ \varphi_{1 3^n}(E) \subset \varphi_{2 1^n}(E)$ for some $n \in \N \cup \{0\}$.
    By (\ref{eq:maximality}) we have $f\circ \varphi_{1 3^n}(E) \not \subset \varphi_{2 1^n 3}(E)$.
    Applying Lemma \ref{lemma:dichotomy} (i) for the map $\varphi_{2 1^{n}}^{-1} \circ f \circ \varphi_{1 3^n}$, we obtain $$f\circ \varphi_{1 3^n}(E) \subset \varphi_{2 1^{n+1}}(E) \cup \varphi_{2 1^n 2} (E).$$
    That is, $$f\circ \varphi_{13^n}\big( \varphi_1(E) \cup \varphi_2(E) \big) \cup f\circ \varphi_{13^{n+1}}(E) \subset \big( \varphi_{2 1^{n+1}}(E) \cup \varphi_{2 1^n 2}(E) \big).$$
    Again by (\ref{eq:maximality}), we have $f\circ \varphi_{13^n}\big( \varphi_1(E) \cup \varphi_2(E) \big) \not\subset \varphi_{2 1^n 2}(E)$.
    It follows that $f\circ \varphi_{13^n}\big( \varphi_1(E) \cup \varphi_2(E) \big) \cap \varphi_{2 1^{n+1}}(E) \ne \emptyset$.
    Note that the set $f\circ \varphi_{13^{n+1}}(E)$ lies in the left of $f\circ \varphi_{13^n}\big( \varphi_1(E) \cup \varphi_2(E) \big)$.
    Thus, we conclude that $f\circ \varphi_{13^{n+1}}(E) \subset \varphi_{2 1^{n+1}} (E)$.
    By induction on $n$, we prove the claim.

  By the claim, we have $$f (\rho) = f\circ \varphi_{1 3^{n}}(1) \le \varphi_{2 1^n} (1) = \varphi_{2}(\rho^n).$$
  Letting $n \to \f$, we have $f (\rho) \le \varphi_{2}(0)=\lambda$, which contradicts (\ref{eq:condition-14}).

  Therefore, we conclude that $-r \big( \varphi_1(E) \cup \varphi_2(E) \big) + t \not\subset E$ for any $0< r < 1$ and for any $t \in \R$.
\end{proof}

\begin{proof}[Proof of Theorem \ref{thm:three-map} (i) for $\rho < \lambda < (1-\rho)/2$]
  Write $f(x) = r x + t$ with $0< |r| < 1$ and $t \in \R$. Note that $f(E) \subset E$.
  By Proposition \ref{prop:no-negative-1}, we obtain $0< r < 1$.
  Then by Proposition \ref{prop:positive-1}, there exists $i_1 \in \{1,2,3\}$ such that $f(E) \subset f_{i_1}(E)$.
  Choose the longest word $\ii \in \Omega^*$ such that \[ f(E) \subset \varphi_{\ii}(E). \]
  It follows that $r \le \rho^{|\ii|}$.
  If $r < \rho^{|\ii|}$, then applying Proposition \ref{prop:positive-1} for the map $\varphi_{\ii}^{-1}\circ f$ we obtain $f(E) \subset \varphi_{\ii i_0}(E)$ for some $i_0 \in \{1,2,3\}$.
  This contradicts the longest length of $\ii$.
  Thus, we have $r=\rho^{|\ii|}$.
  By Lemma \ref{lemma:subset-to-equal}, we obtain $f(E) = \varphi_{\ii}(E)$. It follows that $f(x) = \varphi_{\ii}(x)$.
\end{proof}

\section{Self-embeddings for $\lambda =\rho$}\label{sec:OSC}

In this section we will prove Theorem \ref{thm:three-map} (i) for $\lambda=\rho$.
We always assume that $\lambda = \rho$ and write $E=E_{\rho,\lambda}$ and $\Omega^* = \Omega_3^*$.
Then we have $\varphi_2(x) = \rho x + \rho$, $G_2 = (2\rho, 1-\rho)$, and \[ [0,1] \setminus E = \bigcup_{\ii \in \Omega^*} \varphi_{\ii}(G_2). \]

\begin{figure}[h!]
\begin{center}
\begin{tikzpicture}[
    scale=10,
    axis/.style={very thick, ->},
    important line/.style={very thick},
    every node/.style={color=black}
    ]

    \draw[{|-|}, important line] (0, 0)--(1, 0);
    \node[above,scale=1pt]at(0,0.02){$0$};
    \node[above,scale=1pt]at(1,0.02){$1$};

    \draw[{|-|}, important line] (0, -0.05)--(0.3, -0.05);
    \node[above,scale=1pt]at(0.15,-0.13){$\varphi_1([0,1])$};
    \draw[{-|}, important line] (0.3, -0.05)--(0.6, -0.05);
    \node[above,scale=1pt]at(0.45,-0.13){$\varphi_2([0,1])$};
    \draw[{|-|}, important line] (0.7, -0.05)--(1, -0.05);
    \node[above,scale=1pt]at(0.85,-0.13){$\varphi_3([0,1])$};

    \node[above,scale=1pt]at(0.65,-0.085){$G_2$};
\end{tikzpicture}
\end{center}
\caption{The first iteration for $E_{\rho,\lambda}$ when $\lambda=\rho$}\label{fig:1}
\end{figure}
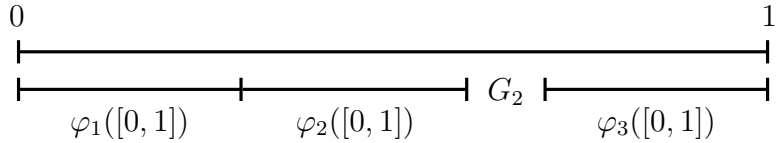

\begin{lemma}\label{lemma:positive-2-1}
  Let $f(x) = r x +t$ with $0< r < 1$ and $t \in \R$. If there exists $\ii \in \Omega^*$ such that $$\varphi_{\ii}(G_2) \subset f(G_2),$$ and if $\varphi_{\ii}(0) \le f(0)$ or $f(1) \le \varphi_{\ii}(1)$, then we have $f(x)= \varphi_{\ii}(x)$.
\end{lemma}
\begin{proof}
  Since $\varphi_{\ii}(G_2) \subset f(G_2)$, we have $r \ge \rho^{|\ii|}$, $f(\lambda+\rho) \le \varphi_{\ii}(\lambda+\rho)$, and $\varphi_{\ii}(1-\rho) \le f(1-\rho)$.
  It follows that \[ f(0) = f(\lambda+\rho) - r(\lambda+\rho) \le \varphi_{\ii}(\lambda+\rho)- r(\lambda+\rho) = \varphi_{\ii}(0) + (\lambda+\rho)(\rho^{|\ii|}-r) \le \varphi_{\ii}(0), \]
  and \[ f(1) = f(1-\rho) + r \rho \ge \varphi_{\ii}(1-\rho) + r \rho = \varphi_{\ii}(1) + \rho(r-\rho^{|\ii|}) \ge \varphi_{\ii}(1). \]
  If $\varphi_{\ii}(0) \le f(0)$ or $f(1) \le \varphi_{\ii}(1)$, then we must have $r = \rho^{|\ii|}$.
  Then we conclude that $\varphi_{\ii}(G_2) = f(G_2)$, which implies that $f(x)=\varphi_{\ii}(x)$.
\end{proof}

\begin{lemma}\label{lemma:positive-2-2}
  Let $f(x) = r x +t$ with $0< r < 1$ and $t \in \R$. If
  \begin{equation}\label{eq:subset-1-2}
    f(E) \subset \varphi_1(E) \cup \varphi_2(E),
  \end{equation}
  and $f(0) < \rho < f(1)$, then we have \[ f(E) \subset \varphi_{13}(E) \cup \varphi_2\big( \varphi_1(E) \cup \varphi_2(E) \big). \]
\end{lemma}
\begin{proof}
  By (\ref{eq:subset-1-2}) it suffices to show \[ f(E) \cap \varphi_1\big( \varphi_1(E) \cup \varphi_2(E) \big)=\emptyset \; \text{and}\; f(E) \cap \varphi_{23}(E) =\emptyset. \]

  If $f(E) \cap \varphi_1\big( \varphi_1(E) \cup \varphi_2(E) \big)\ne\emptyset$, then $f(0) \le \varphi_1(2\rho)$. Note that $f(1)>\rho =\varphi_1(1)$.
  Then the open interval $\varphi_1(G_2)$ is contained in gaps of $f(E)$. That is,
  \[ \varphi_1(G_2) \subset \bigcup_{\ii \in \Omega^*} f \circ \varphi_{\ii}(G_2). \]
  For any $\ii \in \Omega^*$ with $|\ii|\ge 1$, we have $|f \circ \varphi_{\ii}(G_2)|= r \rho^{|\ii|} |G_2| < \rho |G_2| = |\varphi_1(G_2)|$.
  Thus, we have $\varphi_1(G_2) \subset f(G_2)$.
  Note by (\ref{eq:subset-1-2}) that $\varphi_1(0) \le f(0)$.
  By Lemma \ref{lemma:positive-2-1}, we obtain $f(x) = \varphi_1(x)$.
  So, $f(1) = \varphi_1(1) = \rho$.
  This leads to a contradiction.

  If $f(E) \cap \varphi_{23}(E) \ne \emptyset$, then $f(1) \ge \varphi_{23}(0) = \varphi_2(1-\rho)$.
  Note that $f(0) < \rho = \varphi_2(0)$.
  Then the open interval $\varphi_2(G_2)$ is contained in gaps of $f(E)$.
  Similarly, we have $\varphi_2(G_2) \subset f(G_2)$.
  Note by (\ref{eq:subset-1-2}) that $\varphi_2(1) \ge f(1)$.
  By Lemma \ref{lemma:positive-2-1}, we obtain $f(x) = \varphi_2(x)$.
  It follows that $f(0) = \varphi_2(0) = \rho$.
  This also leads to a contradiction.
\end{proof}

\begin{proposition}\label{prop:positive-2}
  Let $f(x) = r x +t$ with $0< r < 1$ and $t \in \R$. If $f(E) \subset E$, then there exists $i \in \{1,2,3\}$ such that $f(E) \subset \varphi_{i}(E)$.
\end{proposition}
\begin{proof}
  Suppose on the contrary that $f(E) \not\subset \varphi_i(E)$ for all $i \in \{1,2,3\}$.
  By Lemma \ref{lemma:dichotomy} (i), we have $$ f(E) \subset \varphi_1(E) \cup \varphi_2(E)\;\;\text{and}\;\; f(0) < \rho < f(1). $$
  By Lemma \ref{lemma:positive-2-2}, we have $f(E) \subset \varphi_{13}(E) \cup \varphi_2\big( \varphi_1(E) \cup \varphi_2(E) \big)$.
  Then we can find $\ell \in \N$ such that
  \begin{equation}\label{eq:ell-subset}
    f(E) \subset \varphi_{13^\ell}(E)\cup \varphi_{21^{\ell-1}}\big( \varphi_1(E) \cup \varphi_2(E) \big),
  \end{equation}
  and
  \begin{equation}\label{eq:ell-not-subset}
    f(E) \not\subset \varphi_{13^{\ell+1}}(E)\cup \varphi_{21^{\ell}}\big( \varphi_1(E) \cup \varphi_2(E) \big).
  \end{equation}
  Thus we have \[ f(0) < \varphi_{13^{\ell+1}}(0) \;\; \text{or}\;\; f(1) > \varphi_{21^{\ell} 2}(1).\]

  If $f(0) < \varphi_{13^{\ell+1}}(0)$, then by (\ref{eq:ell-subset}) we have $f(0)\le \varphi_{13^{\ell}2}(1)$. Note that $f(1) > \rho=\varphi_1(1)$.
  Then the open interval $\varphi_{13^\ell}(G_2)$ must be contained in gaps of $f(E)$. That is,
  \[ \varphi_{13^\ell}(G_2) \subset \bigcup_{\ii \in \Omega^*} f \circ \varphi_{\ii}(G_2). \]
  By (\ref{eq:ell-subset}) we have $r \le 3 \rho^{\ell+1}< \rho^{\ell}$ because $\rho < 1/3$.
  It follows that $|\varphi_{13^\ell}(G_2)| = \rho^{\ell+1} |G_2| > r \rho |G_2|$.
  Thus we have $\varphi_{13^{\ell}}(G_2) \subset f(G_2)$.
  Note again by (\ref{eq:ell-subset}) that $\varphi_{13^{\ell}}(0) \le f(0)$.
  By Lemma \ref{lemma:positive-2-1}, we obtain $f(x)= \varphi_{13^{\ell}}(x)$.
  It follows that $f(1) = \varphi_{13^\ell}(1) = \rho$, a contradiction.

  If $f(0) \ge \varphi_{13^{\ell+1}}(0)$ and $f(1) > \varphi_{21^{\ell} 2}(1)$, then by (\ref{eq:ell-subset}) we have $f(1) \ge \varphi_{21^{\ell} 3}(0)$. Note that $f(0) < \rho=\varphi_2(0)$. Then the open intervals $\varphi_{2 1^{\ell+1}}(G_2)$ and $\varphi_{21^\ell}(G_2)$ are contained in gaps of $f(E)$.
  That is,
  \[ \varphi_{2 1^{\ell+1}}(G_2) \cup \varphi_{21^\ell}(G_2) \subset \bigcup_{\ii \in \Omega^*} f \circ \varphi_{\ii}(G_2). \]
  Since $r < \rho^\ell$, we have $|\varphi_{21^\ell}(G_2)| = \rho^{\ell+1} |G_2| > r \rho |G_2|$.
  Thus, we obtain
  \begin{equation}\label{eq:condition-12}
    \varphi_{21^\ell}(G_2) \subset f(G_2).
  \end{equation}
  Note that $|\varphi_{2 1^{\ell+1}}(G_2)| > r \rho^2 |G_2|$ and the interval $\varphi_{2 1^{\ell+1}}(G_2)$ lies in the left of $\varphi_{21^\ell}(G_2)$.
  So we have
  \begin{equation}\label{eq:condition-11}
    \varphi_{2 1^{\ell+1}}(G_2) \subset f\circ\varphi_1(G_2) \cup f\circ\varphi_2(G_2) \cup f(G_2).
  \end{equation}
  By (\ref{eq:condition-12}) we have $r \ge \rho^{\ell+1}$.
  If $r = \rho^{\ell+1}$, then $\varphi_{21^\ell}(G_2)= f(G_2)$. This implies that $f(x) = \varphi_{21^\ell}(x)$. So we have $f(0) = \varphi_{21^\ell}(0)=\rho$, a contradiction.
  Thus, we obtain $r > \rho^{\ell+1}$.
  Again by (\ref{eq:condition-12}) we have $f(2\rho) \le \varphi_{21^\ell}(2\rho)$.
  It follows that
  \begin{align*}
    f\circ\varphi_1(1-\rho) & = f(\rho -\rho^2) = f(2\rho) - r (\rho+\rho^2)\\
     & < \varphi_{21^\ell}(2\rho) - \rho^{\ell+1}(\rho+\rho^2) \\
     & = \rho + \rho^{\ell+2} - \rho^{\ell+3} \\
     & = \varphi_{21^{\ell+1}}(1-\rho),
  \end{align*}
  which means that $\varphi_{2 1^{\ell+1}}(G_2) \not\subset f\circ\varphi_1(G_2)$.
  Together with (\ref{eq:condition-11}) we have
  \[\varphi_{2 1^{\ell+1}}(G_2) \subset f\circ\varphi_2(G_2) \cup f(G_2).\]
  It follows that $f\circ \varphi_2(2\rho) \le \varphi_{2 1^{\ell+1}}(2\rho)$.
  Then we have
  \begin{align*}
    f(0) & = f(\rho+2\rho^2) - r(\rho+2\rho^2) =  f\circ \varphi_2(2\rho) - r(\rho+2\rho^2) \\
    & < \varphi_{2 1^{\ell+1}}(2\rho) - \rho^{\ell+1}(\rho+2\rho^2) \\
    & = \rho - \rho^{\ell+2} \\
    & = \varphi_{13^{\ell+1}}(0),
  \end{align*}
  which contradicts the previous assumption that $f(0) \ge \varphi_{13^{\ell+1}}(0)$.

  Therefore, we conclude that there exists $i \in \{1,2,3\}$ such that $f(E) \subset \varphi_i(E)$.
\end{proof}

\begin{lemma}\label{lemma:no-negative}
  For any $0 < r < 1$, we have
  \[ -r E + \frac{\rho}{1-\rho} \not\subset \varphi_1(E) \cup \varphi_2(E). \]
\end{lemma}
\begin{proof}
  Suppose on the contrary that there exists $0< r_0 < 1$ such that
  \begin{equation}\label{eq:r-0}
    -r_0 E + \frac{\rho}{1-\rho} \subset \varphi_1(E) \cup \varphi_2(E).
  \end{equation}
  Let $r$ be the maximal value $0< r_0 \le 1$ satisfying (\ref{eq:r-0}).
  Write \[ f(x) = -r x + \frac{\rho}{1-\rho}. \]
  Then we have $f(E) \subset \varphi_1(E) \cup \varphi_2(E)$ and $r < 1$.
  Note that the point $\rho/(1-\rho)$ is the fixed point of $\varphi_2$.

  If $f(1)\ge \rho$, then noting that $f(0) = \rho/(1-\rho)$ we have \[ f(E) \subset \big( \varphi_1(E) \cup \varphi_2(E) \big) \cap \Big[ \rho, \frac{\rho}{1-\rho} \Big] \subset \varphi_2\big( \varphi_1(E) \cup \varphi_2(E) \big). \]
  This implies that \[ -\frac{r}{\rho} E + \frac{\rho}{1-\rho} \subset \varphi_1(E) \cup \varphi_2(E), \]
  which contradicts the maximality of $r$.
  Thus, we obtain $f(1) < \rho = \varphi_2(0)$.

  Note that $f(0)= \rho/(1-\rho) > \varphi_{21}(1)$. Then the open interval $\varphi_{21}(G_2)$ is contained in gaps of $f(E)$.
  It follows that \[ \varphi_{21}(G_2) \subset f \circ \varphi_1(G_2) \cup f \circ \varphi_2(G_2) \cup f(G_2) \cup f\circ\varphi_3(G_2). \]

  \textbf{Case (i)}: $\varphi_{21}(G_2) \subset f \circ \varphi_1(G_2)$.
  Then we have $r\ge \rho$ and $f\circ \varphi_1(2\rho) \ge \varphi_{21}(1-\rho)$.
  It follows that \[ f\circ \varphi_{12}(0) = f\circ \varphi_{1}(\rho) = f\circ \varphi_1(2\rho) + r\rho^2 \ge \varphi_{21}(1-\rho)+ \rho^3 = \varphi_{21}(1) = \varphi_{213}(1), \]
  and \[ f\circ \varphi_{12}(1) = f(2\rho^2) = \frac{\rho}{1-\rho} - 2r\rho^2 < \varphi_{2222}(1) - 2\rho^3 = \varphi_{213}(2\rho). \]
  This implies that the open interval $\varphi_{213}(G_2)$ is contained in gaps of $f\circ \varphi_{12}(E)$.
  Then we have $\varphi_{213}(G_2) \subset f\circ \varphi_{12}(G_2)$, and hence,
  \begin{equation}\label{eq:condition-15}
    f(\rho^2+2\rho^3) = f\circ \varphi_{12}(2\rho) \ge \varphi_{213}(1-\rho) = \rho+\rho^2 - \rho^4.
  \end{equation}
  Note that \[ f\circ \varphi_{11}(1) = f(\rho^2) = \frac{\rho}{1-\rho} - r\rho^2 < \varphi_{2222}(1) - \rho^3 = \varphi_{221}(2\rho), \]
  and \[ f\circ \varphi_{11}(0) = f(0) = \frac{\rho}{1-\rho} > \varphi_{222}(0) = \varphi_{221}(1).\]
  Then the open interval $\varphi_{221}(G_2)$ is contained in gaps of $f\circ \varphi_{11}(E)$.
  It follows that $\varphi_{221}(G_2) \subset f\circ \varphi_{11}(G_2)$, and hence,
  \[ f(\rho^2 - \rho^3) = f\circ \varphi_{11}(1-\rho) \le \varphi_{221}(2\rho) =\rho^2 +\rho + 2\rho^4. \]
  Together with (\ref{eq:condition-15}), we have \[ 2r\rho^3 = f(\rho^2 - \rho^3) - f(\rho^2+2\rho^3) \le 3\rho^4, \]
  i.e., $r \le \rho$. So, we obtain $r=\rho$, and hence $\varphi_{21}(G_2) = f \circ \varphi_1(G_2)$.
  Then we have $f\circ \varphi_1(2\rho) = \varphi_{21}(1-\rho)$, which implies that \[ r = \frac{\rho(2-\rho)}{2(1-\rho)} > \rho. \]
  This leads to a contradiction

  \textbf{Case (ii)}: $\varphi_{21}(G_2) \subset f \circ \varphi_2(G_2)$.
  Then we have $r \ge \rho$ and $f\circ \varphi_2(2\rho) \ge \varphi_{21}(1-\rho)$,
  which implies that \[ r \le \frac{\rho^2(2-\rho)}{(1-\rho)(2\rho+1)}=\rho \frac{2\rho - \rho^2}{1+\rho-2\rho^2} < \frac{\rho}{2},\]
  where the last inequality follows from $0< \rho < 1/3$.
  This leads to a contradiction.

  \textbf{Case (iii)}: $\varphi_{21}(G_2) \subset f(G_2) \cup f \circ \varphi_3(G_2)$.
  It follows that  $\varphi_{21}(G_2) \subset \big( f(1), f(2\rho) \big)$.
  Then we have $f(2\rho) \ge \varphi_{21}(1-\rho)$,
  which implies that \[ r \le \frac{\rho^2(2-\rho)}{2(1-\rho)} < \frac{\rho^2}{1-\rho}. \]
  Note that \[ f(1) = \frac{\rho}{1-\rho} - r < \rho.\]
  It follows that \[ r > \frac{\rho^2}{1-\rho}, \]
  a contradiction.

  We have deduced a contradiction for all cases.
  Therefore, we conclude the desired result.
\end{proof}

\begin{proposition}\label{prop:no-negative-2}
  For any $0 < r < 1$ and for any $t \in \R$, we have
  \[ -r E + t \not\subset \varphi_1(E) \cup \varphi_2(E). \]
\end{proposition}
\begin{proof}
  Suppose on the contrary that there exists $0< r < 1$ and $t_0 \in \R$ such that \[ -r E + t_0 \subset \varphi_1(E) \cup \varphi_2(E). \]
  We claim that
  \begin{equation}\label{eq:r-0-rho}
    \text{there exists}\; t_1 \in \R\;\text{such that}\; - \frac{r}{\rho} E + t_1 \subset \varphi_1(E) \cup \varphi_2(E).
  \end{equation}
  If the claim (\ref{eq:r-0-rho}) were proved, then by induction, for all $n \in \N$ there would exist $t_n \in \R$ such that \[- \frac{r}{\rho^n} E + t_n \subset \varphi_1(E) \cup \varphi_2(E).\]
  Obviously, this leads to a contradiction.
  It remains to show the claim (\ref{eq:r-0-rho}).

  By Lemma \ref{lemma:no-negative}, we have $t_0 \ne \rho/(1-\rho)$.
  Let
  \begin{equation}\label{eq:def-b}
    b =
    \left\{
    \begin{aligned}
      \max\big\{  t \in \R: -r E + t \subset \varphi_1(E) \cup \varphi_2(E) \big\},  \quad\text{if}\;\; t_0 > \frac{\rho}{1-\rho}; \\
      \min\big\{  t \in \R: -r E + t \subset \varphi_1(E) \cup \varphi_2(E) \big\},  \quad\text{if}\;\; t_0 < \frac{\rho}{1-\rho}.
    \end{aligned}
    \right.
  \end{equation}
  Write $f(x) = -rx + b$. Then we have $f(E)\subset \varphi_1(E) \cup \varphi_2(E)$.

  Let
  \begin{align*}
    F_1 & = \varphi_{11}(E) \cup \varphi_{12}(E), \\
    F_2 & = \varphi_{13}(E) \cup \varphi_{21}(E) \cup \varphi_{22}(E), \\
    F_3 & = \varphi_{23}(E).
  \end{align*}
  Then $f\big( \varphi_1(E) \cup \varphi_2(E) \big) \subset \varphi_1(E) \cup \varphi_2(E) = F_1 \cup F_2 \cup F_3$.
  Note that \[ \min_{1\le i < j \le 3} \mathrm{dist}(F_i,F_j) = \rho |G_2| > r \rho |G_2| = \mathfrak{g}\big( f(\varphi_1(E) \cup \varphi_2(E)) \big). \]
  By Lemma \ref{gap-lemma}, there exists $i_0 \in \{1,2,3\}$ such that
  \[ f\big( \varphi_1(E) \cup \varphi_2(E) \big) \subset F_{i_0}.\]

  \textbf{Case (i)}: $f\big( \varphi_1(E) \cup \varphi_2(E) \big) \subset F_1=\varphi_{11}(E) \cup \varphi_{12}(E)$. Then we have $f(0)\le \varphi_{12}(1)$. Thus, we obtain
  \begin{align*}
  f(E) & \subset \big( \varphi_1(E) \cup \varphi_2(E) \big) \cap (-\f, f(0)] \\
  & \subset \big( \varphi_1(E) \cup \varphi_2(E) \big) \cap (-\f, \varphi_{12}(1)] \\
  & = \rho\big( \varphi_{1}(E) \cup \varphi_{2}(E) \big),
  \end{align*}
  which implies the claim (\ref{eq:r-0-rho}).

  \textbf{Case (ii)}: $f\big( \varphi_1(E) \cup \varphi_2(E) \big) \subset F_2 = \varphi_{13}(E) \cup \varphi_{21}(E) \cup \varphi_{22}(E)$.
  If $f\circ \varphi_1(E) \subset \varphi_{21}(E) \cup \varphi_{22}(E)$, then we have \[ \varphi_2^{-1} \circ f \circ \varphi_1(E) = -r E + \frac{b}{\rho}-1 \subset \varphi_1(E) \cup \varphi_2(E).  \]
  If $t_0 > \rho/(1-\rho)$, then by (\ref{eq:def-b}) we have $b > \rho/(1-\rho)$, i.e., $b/\rho -1 > b$, which contradicts the maximality of $b$;
  if $t_0 < \rho/(1-\rho)$, then by (\ref{eq:def-b}) we have $b < \rho/(1-\rho)$, i.e., $b/\rho -1 < b$, which contradicts the minimality of $b$.
  Thus, we conclude that \[ f\circ \varphi_1(E) \cap \varphi_{13}(E) \ne \emptyset. \]
  Then we have $f(\rho) = f\circ \varphi_1(1) \le \varphi_{13}(1)$.
  This implies that \[ f\circ \varphi_2(E) \subset F_2 \cap (-\f, f(\rho)] \subset F_2 \cap (-\f, \varphi_{13}(1)]= \varphi_{13}(E). \]
  It follows that $r \le \rho$.
  If $r = \rho$, then by Lemma \ref{lemma:subset-to-equal} we have $\varphi_{13}^{-1} \circ f \circ \varphi_2(E) = E$.
  This means that the set $E$ is symmetric, which contradicts Lemma \ref{lemma:symmetric}.
  So, we obtain $r < \rho$.
  Thus, \[ \mathfrak{g}\big( f(E) \big) = r|G_2| < \rho|G_2| = \min_{1\le i < j \le 3} \mathrm{dist}(F_i,F_j). \]
  By Lemma \ref{gap-lemma}, we conclude that \[ f(E) \subset F_2= \varphi_{13}(E) \cup \varphi_{21}(E) \cup \varphi_{22}(E). \]
  Since $r<\rho$ and $f(\rho) = f\circ \varphi_1(1) \le \varphi_{13}(1)$, we have \[ f(0) = f(\rho) + r\rho < \varphi_{13}(1) + \rho^2 = \varphi_{21}(1).\]
  It follows that
  \begin{align*}
    f(E) & \subset F_2 \cap (-\f,f(0)]  \\
    & \subset F_2 \cap (-\f,\varphi_{21}(1)] \\
    & = \varphi_{13}(E) \cup \varphi_{21}(E) \\
    & = \rho\big( \varphi_1(E) \cup \varphi_2(E) \big) + \rho - \rho^2,
  \end{align*}
  which implies the claim (\ref{eq:r-0-rho}).

  \textbf{Case (iii)}: $f\big( \varphi_1(E) \cup \varphi_2(E) \big) \subset F_3=\varphi_{23}(E)$.
  It follows that $2r\rho \le \rho^2$, which implies $r < \rho$.
  Thus, \[ \mathfrak{g}\big( f(E) \big) = r|G_2| < \rho|G_2| = \min_{1\le i < j \le 3} \mathrm{dist}(F_i,F_j). \]
  By Lemma \ref{gap-lemma}, we conclude that \[ f(E) \subset F_3= \varphi_{23}(E) = \rho \cdot \varphi_1(E) + \rho(2-\rho) , \]
  which also implies the claim (\ref{eq:r-0-rho}).

  Therefore, we have proved the claim and complete the proof.
\end{proof}

\begin{proof}[Proof of Theorem \ref{thm:three-map} (i) for $\lambda=\rho$]
  Write $f(x) = r x + t$ with $0< |r| < 1$ and $t \in \R$.
  Note that $f(E) \subset E$. It follows that $\rho\cdot f(E) \subset \rho E = \varphi_1(E)$. By Proposition \ref{prop:no-negative-2}, we obtain $0< r < 1$.
  Then by Proposition \ref{prop:positive-2}, there exists $i_1 \in \{1,2,3\}$ such that $f(E) \subset f_{i_1}(E)$.
  Choose the longest word $\ii \in \Omega^*$ such that \[ f(E) \subset \varphi_{\ii}(E). \]
  It follows that $r \le \rho^{|\ii|}$.
  If $r < \rho^{|\ii|}$, then applying Proposition \ref{prop:positive-2} for the map $\varphi_{\ii}^{-1}\circ f$ we obtain $f(E) \subset \varphi_{\ii i_0}(E)$ for some $i_0 \in \{1,2,3\}$.
  This contradicts the longest length of $\ii$.
  Thus, we have $r=\rho^{|\ii|}$.
  By Lemma \ref{lemma:subset-to-equal}, we obtain $f(E) = \varphi_{\ii}(E)$. It follows that $f(x) = \varphi_{\ii}(x)$.
\end{proof}

\section*{Acknowledgements}
The author was supported by the National Natural Science Foundation of China (No. 12501110, 12471085) and the China Postdoctoral Science Foundation (No. 2024M763857).


\begin{thebibliography}{10}

\bibitem{Algom-2020}
A.~Algom.
\newblock Affine embeddings of {C}antor sets in the plane.
\newblock {\em J. Anal. Math.}, 140(2):695--757, 2020.

\bibitem{Algom-Hochman-2019}
A.~Algom and M.~Hochman.
\newblock Self-embeddings of {B}edford-{M}c{M}ullen carpets.
\newblock {\em Ergodic Theory Dynam. Systems}, 39(3):577--603, 2019.

\bibitem{Dajani-Kong-Yao-2019}
K.~Dajani, D.~Kong, and Y.~Yao.
\newblock On the structure of {$\lambda$}-{C}antor set with overlaps.
\newblock {\em Adv. in Appl. Math.}, 108:97--125, 2019.

\bibitem{Deng-Lau-2013}
Q.-R. Deng and K.-S. Lau.
\newblock On the equivalence of homogeneous iterated function systems.
\newblock {\em Nonlinearity}, 26(10):2767--2775, 2013.

\bibitem{Deng-Lau-2017}
Q.-R. Deng and K.-S. Lau.
\newblock Structure of the class of iterated function systems that generate the
  same self-similar set.
\newblock {\em J. Fractal Geom.}, 4(1):43--71, 2017.

\bibitem{Deng-Yao-2022}
Q.-R. Deng and Y.-H. Yao.
\newblock On the group of isometries of planar {IFS} fractals.
\newblock {\em Nonlinearity}, 35(1):445--469, 2022.

\bibitem{Elekes-Keleti-Mathe-2010}
M.~Elekes, T.~Keleti, and A.~M\'ath\'e.
\newblock Self-similar and self-affine sets: measure of the intersection of two
  copies.
\newblock {\em Ergodic Theory Dynam. Systems}, 30(2):399--440, 2010.

\bibitem{Feng-Wang-2009}
D.-J. Feng and Y.~Wang.
\newblock On the structures of generating iterated function systems of {C}antor
  sets.
\newblock {\em Adv. Math.}, 222(6):1964--1981, 2009.

\bibitem{Hochman-2017}
M.~Hochman.
\newblock Some problems on the boundary of fractal geometry and additive
  combinatorics.
\newblock In {\em Recent developments in fractals and related fields}, Trends
  Math., pages 129--174. Birkh\"auser/Springer, Cham, 2017.

\bibitem{Hutchinson-1981}
J.~E. Hutchinson.
\newblock Fractals and self-similarity.
\newblock {\em Indiana Univ. Math. J.}, 30(5):713--747, 1981.

\bibitem{Kong-Yao-2021}
D.~Kong and Y.~Yao.
\newblock On a kind of self-similar sets with complete overlaps.
\newblock {\em Acta Math. Hungar.}, 163(2):601--622, 2021.

\bibitem{Moran-2012}
M.~Mor\'an.
\newblock The group of isometries of a self-similar set.
\newblock {\em J. Math. Anal. Appl.}, 392(1):89--98, 2012.

\bibitem{Xiao-2024}
J.-C. Xiao.
\newblock On a self-embedding problem for self-similar sets.
\newblock {\em Ergodic Theory Dynam. Systems}, 44(10):3002--3011, 2024.

\bibitem{Yao-2015}
Y.~Yao.
\newblock Generating iterated function systems of some planar self-similar
  sets.
\newblock {\em J. Math. Anal. Appl.}, 421(1):938--949, 2015.

\bibitem{Yao-2017}
Y.~Yao.
\newblock Generating iterated function systems for self-similar sets with a
  separation condition.
\newblock {\em Fund. Math.}, 237(2):127--133, 2017.

\bibitem{Yao-Li-2015}
Y.~Yao and W.~Li.
\newblock Generating iterated function systems for a class of self-similar sets
  with complete overlap.
\newblock {\em Publ. Math. Debrecen}, 87(1-2):23--33, 2015.

\bibitem{Yao-Li-2016}
Y.~Yao and W.~Li.
\newblock Generating iterated function systems for the {V}icsek snowflake and
  the {K}och curve.
\newblock {\em Amer. Math. Monthly}, 123(7):716--721, 2016.

\end{thebibliography}

\end{document}